\documentclass[11pt]{amsart}

\usepackage{amsmath}
\usepackage{amsfonts}
\usepackage{amssymb}
\usepackage{eucal}

\evensidemargin\oddsidemargin
\def\ker{\mathop{\rm Ker}\nolimits}

\def\dsum{\displaystyle\sum}

\def\dfrac{\displaystyle \frac}

\def\inc{\mathop{\rm \iota}\nolimits}

\newtheorem{theorem}{Theorem}

\newcounter{prop}
\newtheorem{proposition}[prop]{Proposition}
\newcounter{coro}
\newtheorem{corollary}[coro]{Corollary}

\begin{document}

\date{\today}

\title{On CR (Cauchy-Riemann) almost cosymplectic manifolds}

\author{Piotr Dacko}

\email{{\tt piotrdacko{\char64}yahoo.com}}



\begin{abstract} In the paper we develop a framework for the alternative way of the study of a local geometry of almost cosymplectic manifolds with K\"alerian leaves. The main idea is to apply the concept of a geometry and analysis of CR manifolds. Locally the almost cosymplectic manifold is modeled on the 'mixed' space $\mathbb{R}\times\mathbb{C}^n$. There is given a complete local description of the underlying almost contact metric structure  in the system of local, mixed - real, complex- coordinates. We also introduce a notion of a canonical Hermitian complex connection in the CR structure of a CR almost cosymplectic manifold. As an example we provide detailed descritpion of almost cosymplectic $(-1,\mu,0)$-spaces.  
 \end{abstract}

\maketitle

\begin{center}
 \small{\it Dedicated to the Memory of Marek Kucharski}
\end{center}

\section{Introduction}
The paper is thought of as a preliminary to the subject of possible applications of the methods coming 
from the geometry of CR manifolds or complex geometry. The starting point is a notion of a CR manifold. 

It is very common nowadays in the geometry of the almost contact metric manifold to impose  conditions of a "CR integrability"
of an almost contact metric structure. One of the first result concerning "CR integrable" almost contact metric structures is the theorem of S. Tanno - the characterization of a CR integrable contact metric manifold
\cite{Tanno}. 

In the settings of almost cosymplectic manifolds CR geometry implicitly appeared in the time when Z. Olszak introduced  quite naturally the notion of the almost cosymplectic manifolds with K\"ahlerian leaves \cite{Olszak2}. Now it is clear that 
these manifolds are exactly CR integrable almost cosymplectic manifolds. And this statement is almost tautological from the point of view of the geometry of CR manifolds. However the main focus was on studying the Riemannian geometry using 
tensor calculus. 

In the presented paper we propose an alternative way to study manifolds with K\"ahlerian leaves, based on the CR geometry. We hope that this alternative  allow us to solve some long-standing and yet unsolved problems. For example we do not know are there exist almost cosymplectic non-cosymplectic manifolds 
of a pointwise-constant $\varphi$-sectional curvature in dimensions $> 3$? Even in the case of manifold with K\"ahlerian leaves the answer is unknown. The other benefit  is that new classes of manifolds appear quite naturally on the base of the local description in the so-called CR charts (Sec. 3). The main point is that locally an almost cosymplectic manifold is modeled on  $\mathbb{R}\times\mathbb{C}^n$ thus we have a 'mixed' local coordinates and an almost contact metric structure $(\varphi,\xi,\eta,g)$ can be described using these 'mixed' coordinates. As a working example we provide such description in very details for almost cosymplectic $(-1,\mu,0)$-spaces with $\mu=const$.

\section{Preliminaries}
\subsection{Almost cosymplectic manifolds}
A manifold $\mathcal{M}$ of an odd dimension, $\dim \mathcal M=2n+1\geqslant 3$ smooth, connected, being endowed with an almost contact metric structure $(\varphi,\xi,\eta,g)$, where $\varphi$ is a $(1,1)$-tensor field, $\xi$ a vector field, $\eta$ a 1-form  and $g$ a Riemannian metric, and the following conditions are satisfied  \cite{Blair}
\begin{equation}
\label{structeq}
\begin{array}{l}
   \varphi^{2} = -I+\eta\otimes\xi, \quad \eta(\xi)=1, \quad \eta(X)=g(X,\xi),\\[+4pt]
   g(\varphi X,\varphi Y)=g(X,Y)-\eta(X)\eta(Y).
\end{array}
\end{equation}
is called an almost contact metric manifold. If the forms  $\eta$ and the fundamental skew-symmetric 2-form $\varPhi(X,Y)=g(\varphi X,Y)$ are both closed the almost contact metric manifold  $\mathcal M$ is called almost cosymplectic \cite{GY,Olszak1}.

An almost contact metric manifold is called normal if \cite{Blair}
\begin{equation}
 \label{prelim-norm}
N_\varphi+2d\eta\otimes\xi=0
\end{equation}
here $N_\varphi$ denotes the Nijenhuis torsion of $\varphi$ defined by
\[
 N_\varphi(X,Y)=\varphi^2[X,Y]+[\varphi X,\varphi Y]-\varphi[\varphi X,Y]-\varphi[X,\varphi Y].
\]

The normal almost cosymplectic manifold is called cosymplectic. From (\ref{prelim-norm}) it follows that an almost cosymplectic manifold
is cosymplectic if and only if the torsion $N_\varphi$ vanishes. In that case the tensor field $\varphi$ is integrable as a $G$-structure,
i.e. there is a suitable atlas of local coordinates charts: the local coefficients of $\varphi$ on each chart are constant functions. Arbitrary
 cosymplectic manifold is locally a  Riemannian product of a real line (an open interval) and a K\"ahler manifold. D.~E.~Blair proved \cite{Blair2} that an
almost contact metric manifold is cosymplectic if and only if
\[
 \nabla\varphi =0,
\]
 for the Levi-Civita connection $\nabla$. S.~I.~Goldberg and K.~Yano studying harmonic forms  proved the following theorem \cite{GY}: an almost cosymplectic manifold is cosymplectic if and only if 
\begin{equation}
 \label{prelim-Rphi}
R(X,Y)\varphi Z = \varphi R(X,Y)Z,
\end{equation}
$R$ the Riemann curvature operator
\[
 R(X,Y) = \nabla_X\nabla_Y-\nabla_Y\nabla_X -\nabla_{[X,Y]}.
\]
That theorem can be viewed as an analogue of a similar theorem for K\"ahler manifolds. Precisely for an almost cosymplectic 
manifold (\ref{prelim-Rphi}) implies $\nabla\varphi=0$ and by the Blair result the manifold is cosymplectic.

An almost cosymplectic manifold $\mathcal M$ carries a canonical foliation 
$\mathcal M = \bigcup\limits_{p\in \mathcal{M}}\mathcal{N}_p$ which corresponds to a completely integrable  distribution defined by $\eta=0$. A leaf $\mathcal{N}_p\subset\mathcal{M}$ can be in a natural way considered as an almost K\"ahler manifold $(\mathcal{N}_p, J, G)$. The almost Hermitian structure $(J,G)$ is given  by 
  \[
      \inc_*\circ J=\varphi\circ \inc_*,\quad\quad G=\inc^*g,
  \]    
$\inc$ denotes the inclusion map $\inc:\mathcal{N}_p\subset \mathcal M$. If $\mathcal M$ is cosymplectic then all leaves are in fact K\"ahler manifolds. However 
the converse is not true. Z. Olszak introduced almost cosymplectic manifolds defined by imposing the following geometric condition: each leaf is a K\"ahler manifold. We have the following characterization \cite{Olszak2}: an almost cosymplectic manifold has K\"ahlerian leaves if and only if
\begin{equation}
\label{kaehleav}
    (\nabla_{X}\varphi)Y = -g(\varphi AX,Y)\xi + \eta(Y)\varphi AX,
\end{equation}
a $(1,1)$-tensor field $A$ is defined by
\begin{equation}
\label{defA}
  AX = - \nabla_X\xi.
\end{equation}

If  the curvature operator of the Levi-Civita connection $R(X,Y)Z$ of an almost cosymplectic manifold  
 satisfies
\begin{equation}
\label{prelim-Rxi}
\begin{array}{c}
R(X,Y)\xi = \eta(Y)PX-\eta(X)PY, \\[+6pt] 
P=\kappa Id+\mu h +\nu A
 ×
\end{array}
\end{equation}
$Id$ the identity tensor, $h = \dfrac{1}{2}\mathcal{L}_\xi \varphi$, $\mathcal{L}_\xi$ the Lie derivative, and $\kappa,\mu,\nu$ are functions such that for the 1-forms 
$d\kappa$, $d\mu$, $d\nu$ we have 
\begin{equation}
\label{prelim-dkappa}
d\kappa\wedge\eta=d\mu\wedge\eta=d\nu\wedge\eta=0,
\end{equation}
then the manifold is called an almost cosymplectic $(\kappa,\mu,\nu)$-space \cite{DO2}. The particular case is when $\kappa$, $\mu$, $\nu$ are constants. Almost cosymplectic manifolds satisfying the condition (\ref{prelim-Rxi}) with $\kappa=const$, $\mu=\nu=0$ were studied in \cite{Dacko}; and with $\kappa$, $\mu=const.$, $\nu=0$ in \cite{Endo1,Endo2,Endo3}. Manifolds with $\kappa=-1$, $\mu=const.$, 
 $\nu=0$ are classified in \cite{DO3}.

Given positive functions $\alpha$, $\beta$ on $\mathcal{M}$ an almost contact metric structure $(\varphi',\xi',\eta',g')$
 defined by
\[
 \varphi'=\varphi,\quad \xi'=\beta^{-1}\xi,\quad \eta'=\beta\eta,\quad g'=\alpha g+(\beta^2-\alpha)\eta\otimes\eta,
\]
is called a D-conformal deformation of the structure $(\varphi,\xi,\eta,g)$.  The  structure $(\varphi',\xi',\eta',g')$ is by itself an almost cosymplectic if and only if $\alpha$ is a constant and the 
function $\beta$ satisfies $d\beta\wedge\eta=0$. 
\begin{proposition}
 \mbox{\rm (\cite{DO2})} If $(\mathcal{M},\varphi,\xi,\eta,g)$  is an almost cosymplectic $(\kappa,\mu,\nu)$-space then its image 
by a D-conformal deformation (assuming an image is almost cosymplectic) 
is a $(\kappa',\mu',\nu')$-space where
\[
 \kappa' = \dfrac{\kappa}{\beta^2},\quad \mu'=\dfrac{\mu}{\beta},\quad \nu'=\dfrac{\nu\beta-d\beta(\xi)}{\beta^2}.
\]
\end{proposition}

So, roughly speaking we may say that the class of almost cosymplectic $(\kappa,\mu,\nu)$-spaces is closed with 
respect to the group of the inner D-conformal deformations. Deformations are inner in the sense that they preserve 
the class of almost cosymplectic manifolds.

 For an almost cosymplectic 
$(\kappa,\mu,\nu)$-space the tensor field $A$ and the functions $\kappa$, $\mu$ and $\nu$ satisfy 
the relations
\begin{equation}
 \begin{array}[]{rcl}
A^2 &=& -\kappa(Id -\eta\otimes\xi), \\
\nabla_\xi A &=& \mu\,h+\nu\,A, \\
d\kappa(\xi) &=& 2\nu\kappa,
 ×
\end{array}
\end{equation}
and these relations are fundamental in the process of the classification of $(\kappa,\mu,\nu)$-spaces. Eg. they 
imply that if $\kappa = const \neq 0$ then the function $\nu$ must vanish identically. Another consequence is 
that arbitrary $(\kappa,\mu,\nu)$-space  is (generally speaking locally) D-conformal to a $(-1,\mu,0)$-space. 
Thus the almost cosymplectic $(-1,\mu,0)$-spaces are of particular importance and they are called the ``model spaces''.

\subsection{CR manifolds (\cite{BER,Jacob,Tai})}
For complex vector fields $Z_1$, $Z_2$ a bracket $[Z_1,Z_2]$ is defined as a (unique) complex vector field such that
\[
 [Z_1,Z_2]f = Z_1(Z_2 f)-Z_2(Z_1 f), 
\]
for arbitrary complex-valued smooth function $f$.
A CR manifold is a pair $(\mathcal M, \mathcal H)$ where $\mathcal M$ is a smooth manifold 
and $\mathcal H$ is a $C^\infty$ complex subbundle of $TM\otimes\mathbb C$ such that:
\begin{equation}
     \mathcal H_p\cap\bar{\mathcal H}_p =0, \quad\quad p\in M,
\end{equation}
and the set of sections $\Gamma(\mathcal H)$ of $\mathcal H$ is closed with respect to the bracket operation.

A  $m=\dim_{\mathbb C}\mathcal H$ is called a CR dimension of a CR manifold $(\mathcal M,\mathcal H)$ and $n-2m$  a CR codimension. 

The examples of CR manifolds are complex manifolds (CR codimension is $0$) and real hypersurfaces in $\mathbb C^n$ (CR codimension is $1$). 
 
Now let $(\mathcal M,\mathcal H)$ be a CR manifold of a hypersurface type, i.e.  CR-codim $=1$. For an imaginary non vanishing $1$-form $\tau$ annihilating $\mathcal H\oplus\bar{\mathcal H}$ (it is possible that $\tau$ is defined only locally), and a vector $z\in \mathcal H_p$ let 
\begin{equation}
\label{dlform}
    z \mapsto L_p(z)=\tau_p([Z,\bar Z]_p),
\end{equation}
where $Z\in \Gamma(\mathcal H)$ is an local extension of $z$. The function $z \mapsto L_p(z)$ is a real quadratic form on $\mathcal H_p$ called the Levi form. The form $L_p$ is defined up to a nonzero scalar for if $\tau$ is replaced by $\tau'=f\tau$, $f(p)\neq 0$ then (\ref{dlform}) yields to $L'_p(z)=f(p) L_p(z)$. If $L_p$ is non-degenerate it is a pseudo-Hermitian form on $\mathcal H_p$. 

A CR-manifold $(\mathcal M,\mathcal H)$ is said to be Levi flat if its Levi form vanishes everywhere. 
\subsection{CR structure of an almost cosymplectic manifold}

Let $\mathcal{M}$ be an almost contact metric manifold and $\mathcal D$ a distribution $\mathcal{D}=Im \,\varphi$, therefore
$\mathcal{D}$ is a field of hyperplanes on $\mathcal{M}$
\begin{equation*}
\mathcal{D}: \mathcal{M}\ni p \mapsto \mathcal{D}_p=\lbrace x\in T_p\mathcal{M}: x = \varphi y\; \text{for}\; y\in T_p\mathcal{M}\rbrace.  
\end{equation*}
Equivalently $\mathcal{D}$ can be defined as the kernel of the form $\eta$. We always have $\ker \eta= Im \,\varphi$
The complexification $\mathcal D\otimes \mathbb C$ splits into a direct sum $\mathcal D'\oplus \mathcal D''$ of $\sqrt{-1}$ and  $-\sqrt{-1}$ eigenspaces of the (complexified) tensor field $\varphi$. For a section $Z\in\Gamma(\mathcal D')$ let $X=\dfrac{1}{2}(Z+\bar Z)$. The conditions $\varphi Z=\sqrt{-1}\hspace{1pt}Z$, $\eta(\varphi Z)=0$ together imply that 
\begin{equation}
\label{hsec}
    Z=X-\sqrt{-1}~\varphi X, \quad\eta(X)=0.
\end{equation}
Conversely given a real vector field $X$, $\eta(X)=0$ (\ref{hsec}) defines a section of $\mathcal D'$. 
    
Let suppose that the pair $(\mathcal{M},\mathcal D')$ is an almost cosymplectic  CR manifold, i.e.
$\mathcal{M}$ is an almost cosymplectic manifold and the complex distribution $\mathcal D'$ defined as above is formally involutive. In order to compute the Levi form $L$ we can take $\tau=-\sqrt{-1}\,\eta$ in (\ref{dlform}) 
\begin{equation*}
\label{lform}
    L_p(z) = -\sqrt{-1}\ \eta([Z,\bar Z]_p)=2\,\eta([X,\varphi X]_p),
\end{equation*}
where $Z=X-\sqrt{-1}\ \varphi X$ and $Z_p=z$. Since $d\eta=0$ and $\eta(X)=\eta(\varphi X)=0$ it follows 
that
\begin{equation*}
    L_p(z)=-4\,d\eta(X_p,\varphi X_p) = 0.
\end{equation*}
Therefore $L$ vanishes identically and $(\mathcal M, \mathcal D')$ is a Levi flat CR manifold (of  the hypersurface type).

Intuitively: a Levi flat CR-manifold is foliated by a family of (real) hypersurfaces, each hypersurface admits a complex structure and this complex structure varies when passing from a hypersurface to another hypersurface according to a differentiability class of the manifold. From this of point view, having in mind the definition of almost cosymplectic manifold with K\"ahlerian leaves, the next proposition is almost tautological
\begin{proposition}
\label{kleavcr}
 The pair $(M,\mathcal D')$ is a (Levi flat) CR manifold if and only if $M$ has K\"ahlerian leaves.
\end{proposition}
In other words an almost cosymplectic CR manifold is exactly  the same concept as an almost cosymplectic manifold with K\"ahlerian leaves.

\section{Analytic almost cosymplectic CR manifolds. The canonical Hermitian structure.}
\subsection{CR charts on an analytic almost cosymplectic CR manifold} 

Suppose $\mathcal{M}$ is a real-analytic almost cosymplectic CR manifold (or a manifold with K\"alerian leaves),  and the tensor fields $\varphi$, $\xi$, $\eta$ and $g$ are assumed to be real-analytic. In a consequence the CR structure $\mathcal D'$ of $\mathcal{M}$ is real-analytic - there are local real-analytic sections spanning $\mathcal D'$. Now, according to the theorem of A.~Andreotti and D.C. Hill (\cite{AndrHill}) there is a local embedding (real-analytic) 
\begin{equation}
 f : \mathcal{M} \rightarrow \mathbb{C}^{n+1},
\end{equation}
($ \dim \mathcal{M} =2n+1$) such that $f(\mathcal M)$ is locally a real-analytic hypersurface. If 
$$
\mathcal{M}\ni q \mapsto p=(z^1,\ldots,z^n, z^{n+1})=f(q) \in f(\mathcal M),
$$  
then there is a maximal complex subspace $\mathcal{H}_p$ in $T_p^{(1,0)}\mathbb{C}^{n+1}$ such that 
\begin{equation*}
   \Re (\mathcal{H}_p\oplus\, \overline{\mathcal{H}}_p) 
= J(T_pf(\mathcal{M}))\cap T_pf(\mathcal{M}),
\end{equation*}
$J$ denotes the canonical complex structure of $\mathbb{C}^{n+1}$, and the complexification of the tangent map $f_*$ 
defines a $\mathbb{C}$-linear isomorphism between complex spaces 
$$
:\quad \mathcal{D'}_q \xrightarrow{f^\mathbb{C}_*} \mathcal{H}_p. 
$$
If $\mathcal{U}_q\subset \mathcal{M}$ is  sufficiently small then $f|_{\mathcal{U}_q}$
 is a diffeomorphism onto its image $f(\mathcal{U}_q)$. Now let consider a real hypersurface in $\mathbb{C}^{n+1}$. That is a set of zeros $\mathcal{S}=r^{-1}(0)$ of a smooth real-valued function $r:\mathbb{C}^{n+1}\rightarrow \mathbb{R}$. We assume that $r$ is regular: $dr\neq 0$ for each point of $\mathcal{S}$. The space $\mathbb{C}^{n+1}$ endowed with its canonical flat K\"ahler metric becomes an Euclidean space $\mathbb{E}^{2n+2}$ 
so let $H$ be the second fundamental form of $\mathcal{S}$ treated as a hypersurface in $\mathbb{E}^{2n+2}$. Let define a real quadratic 
differential form on a complex subbundle  $\mathcal{H}:p\mapsto \mathcal{H}_p$, $p\in \mathcal{S}$ ($\mathcal{H}_p$ is defined similarly as above for $f(\mathcal{M})=\mathcal{S}$) by the formula
\[
 L(z) = H(X,X)+ H(JX,JX), \qquad z\in \mathcal{H}_p,\; z= X-\sqrt{-1}JX,\; X\in T_p\mathcal{S}, 
\]
for a complex vector $z$ tangent to $\mathcal{S}$. The form $L$ is a Levi form of the real hypersurface $\mathcal{S}$ (cf. \cite{Tai}).  Let consider two examples
\begin{enumerate}
 \item $\mathcal{S}=\mathbb{S}^{2n+1}(r)$ a canonical sphere of the radius $r$; as the second fundamental form is non-degenerate and definite  the Levi form is nondegenerate and definite: $\mathbb{S}^{2n+1}(r)$ is an example of a strictly pseudo-convex real hypersurface,
\item $\mathcal{S}= \lbrace  p=(z^1,\ldots,z^{n+1})\in \mathbb{C}^{2n+1}: \Im\, z^{n+1} = 0\rbrace$; clearly $\mathcal{S}$ now is simply a hyperplane hence $H=0$ identically, in a consequence the Levi form vanishes; $\mathcal{S}$ is the simplest example of a Levi flat real hypersurface. 
\end{enumerate}
We need the following basic result: a Levi flat real-analytic hypersurface is locally biholomorphic to the hyperplane 
described in the example (2). So let assume that a Levi flat real hypersurface is passing through the origin $o$ of $\mathbb{C}^{n+1}$ then there is a small disk $D_o$ centered at $o$ and a biholomorphism    
$F: D_o \rightarrow D_o$ such that 
\begin{equation*}
 F(\mathcal{S}\cap D_o) = \lbrace \Im\, z^{n+1} = 0 \rbrace \cap D_o.
\end{equation*}
Now let consider the sequence of maps  
\begin{equation}
 \mathcal{U}_q \xrightarrow{f} f(\mathcal{U}_q) \cap D_o \xrightarrow{F} \lbrace \Im\, z^{n+1} = 0 \rbrace \cap D_o,
\end{equation}
the existence of $F$ follows from the fact that $f(\mathcal{U}_q)$ 
is a real hypersurface (for sufficiently small $\mathcal{U}_q$) as it is defined above, i.e. there is a regular real function 
$r$ such that $f(\mathcal{U}_q) \subset r^{-1}(0)$ and $f(\mathcal{U}_q)$ is Levi flat.
\begin{proposition}
 Let $(\mathcal{M}, \mathcal{D}')$ be a real-analytic almost cosymplectic CR manifold. Then each point $q$ of $\mathcal{M}$ 
admits a neighborhood $\mathcal{U}_q$ and a local diffeomorphism $f_q: \mathcal{U}_q \rightarrow (-a,a)\times D'$
where $(-a,a)$ is an open interval $a>0$ and $D'$ is a small disk in $\mathbb{C}^{n}$.
\end{proposition}
 \begin{proof}
  By the Andreotti, Hill theorem there is a local embedding $f$ such that  $f(\mathcal{U}_q)$ is a Levi flat real-analytic hypersurface in $\mathbb{C}^{n+1}$. Now (we may assume $f(q) = o$) there is a local biholomorphism $F$ of an open disk $D_o$ which 
maps $f(\mathcal{U}_q) \cap D_o$ onto $\lbrace \Im\, z^{n+1}=0 \rbrace \cap D$. So an image of $\mathcal{U}_q$ by the 
composition $F\circ f$ is contained in the hyperplane $\Im\, z^{n+1}=0$. In the natural manner we identify 
$\Im\, z^{n+1}=0$ with $\mathbb{R}\times \mathbb{C}^{n}$
\[
 \lbrace \Im\, z^{n+1}=0\rbrace \ni p=(z^1,\ldots,z^{n+1}) \mapsto (t, z^1, z^2,\ldots, z^{n}), \quad t=\Re \,z^{n+1}.
\]
The differential $(F\circ f)_*$ is non-degenerate at $q$. By the standard inverse function arguments we can assert that 
there is an open set of the form $(-a,a)\times D' \subset (F\circ f) (\mathcal{U}_q)$ with well-defined an inverse map
$(-a,a)\times D' \rightarrow \mathcal{U}_q$. An image $\mathcal{U}'_q$ by this inverse map is the required neighborhood and we set $f_q = F\circ f|_{\mathcal{U}'_q}$.
 \end{proof}
The CR structure of $(-a,a)\times D'$ is defined by $T^{(1,0)}D'$ - the complex bundle of $(1,0)$ vector fields on $D'$ in natural way embedded into  the complex tangent bundle of $(-a,a)\times D'$.
The family $(\mathcal{U}_q, f_q)$, $q\in \mathcal{M}$ defines a very particular atlas on $\mathcal{M}$. So the manifold is covered by the coordinates charts of the form 
$$
(t, z^1,z^2,\ldots,z^n): (-a,a)\times D'\rightarrow \mathcal{U}_q, \quad t\in \mathbb{R}, z^i\in \mathbb{C}.
$$ 
The transition functions  
\[
f_{qp} = f_q\circ f_p^{-1}: f_p(\mathcal{U}_p\cap\mathcal{U}_q) \rightarrow f_q(\mathcal{U}_p\cap\mathcal{U}_q) 
\]
(we verify this directly) are given as follows
\[
\begin{array}{rcl}
f_{qp} &:& (t,z^1,\ldots,z^n) \mapsto (t', z^{'1},\ldots, z^{'n}), \\[+4pt]
t' &=& t'(t), \\[+4pt]
z^{'1} &=& z^{'1}( t,z^1,\ldots, z^n), \\
\textellipsis \\
z^{'n} &=& z^{'n}(t,z^1,\ldots,z^n),
\end{array}
\]
If $Z_1,\ldots,Z_n$ are  sections of $\mathcal{D'}$  and such that
in the local coordinates on $\mathcal{U}_q$, $Z_i$'s are 'base' vector fields:
\[
 Z_1 = \dfrac{\partial }{\partial z'^1},\; \ldots \;, Z_n= \dfrac{\partial}{\partial z'^n},
\]
in the local coordinates on $\mathcal{U}_p$, $Z_i$'s have the  decompositions:
\[
 \begin{array}{l}
Z_1 = \sum_{i=1}^n f_1^i\dfrac{\partial}{\partial z^i}, \\
\textellipsis \\
Z_n = \sum_{i=1}^nf_n^i\dfrac{\partial}{\partial z^i}
 ×
\end{array}
\]
$f_j^i$ are complex functions, therefore on $\mathcal{U}_q\cap\mathcal{U}_p$
\[
 \begin{array}{l}
\dfrac{\partial}{\partial z'^1} = \sum_{i=1}^n f_1^i\dfrac{\partial}{\partial z^i}, \\
\textellipsis\\
\dfrac{\partial }{\partial z'^n} = \sum_{i=1}^n f_n^i\dfrac{\partial}{\partial z^i}.
 ×
\end{array}
\]
These identities imply
\[
\dfrac{\partial \bar z^j}{\partial  z'^i}=0, \quad \dfrac{\partial t }{\partial z'^j} =0, \quad i,j =1,\ldots,n.
\]
hence $\dfrac{\partial z^j}{\partial \bar z'^i} = 0$ and as $t$ is a real-valued function $\dfrac{\partial t}{\partial \bar z'^j}=0$.

\subsection{A local description of an almost contact metric structure}
 Once we have a clear idea how to understand a 'complex coordinate' on an almost cosymplectic CR manifold we may describe   the almost cosymplectic structure $(\varphi,\xi,\eta,g)$ in terms of these complex  coordinates in the way similar as it is done in the complex 
geometry of Hermitian manifolds. Here we follow (with necessary changes) the monograph \cite{KobNomV2}.

From now on we will consider  {\em complexified} structure $(\varphi^{\mathbb{C}}, \xi^{\mathbb{C}}, \eta^{\mathbb{C}},g^{\mathbb{C}})$ however we use the same notation for both the structure and its complexification. It should be clear from the context when the structure is in fact the complexification. For example $\xi^{\mathbb{C}}$ is a complex vector field. i.e. a section of $T_{\mathbb{C}}\mathcal{M}=T\mathcal{M}\otimes\mathbb{C}$.
The following convention is assumed:  indices matching lowercase Latin letters  $i,j,k,l,\dots$ run from 1 to $n$, while Latin capitals $A, B, C,\ldots$ run through $0,1,\ldots,n,\bar 0,\bar 1,\ldots,\bar n$, moreover 
$\bar 0 = 0$. We set
\begin{equation}
\label{defZ}
     Z_0 = \dfrac{\partial}{\partial t},\quad\quad
     Z_1=\dfrac{\partial}{\partial z^1},\;\ldots\;, Z_n= \dfrac{\partial}{\partial z^n}, 
\end{equation}     
and 
\[
Z_{\bar i} = \overline{Z_i } = \dfrac{\partial}{\partial \bar z^i}. 
\]

\begin{equation}
\label{cstruct}
     \varphi Z_A =\varphi_A^B Z_B ,\quad\quad \xi = \dsum_A\xi^A Z_A, \quad\quad\eta_A=\eta(Z_A),\quad\quad
     g_{AB}=g(Z_A,Z_B).
\end{equation}
Note that $g_{\bar A \bar B}=\bar g_{AB}$. Since the vector fields of $Z_1,\ldots,Z_n$ span $\mathcal D'$ we have 
\begin{equation}
\label{fiza}  \varphi Z_i =\sqrt{-1}Z_i,\quad  \varphi Z_{\bar i} =-\sqrt{-1}Z_{\bar i}.
\end{equation}    
Similarly to a Hermitian metric on a complex manifold
\begin{equation}
\label{isotr}
     g_{i\,j}= g_{\bar i\, \bar j}=0,
\end{equation}
and $(g_{i\,\bar j})$ is a $n\times n$ Hermitian matrix $\bar g_{i\, \bar j} = g_{j\,\bar i}$.
The coefficients $b_A=g(Z_0,Z_A)$ satisfy
\begin{equation}
\label{coisotr}
   b_{00}=\bar b_{00},\quad  b_{\bar i} = g_{0\bar i}=g_{\bar 0\bar i}=\bar b_i.              
\end{equation}
Summing up all above formulas we can  write
\begin{equation}
\label{metr}
  ds^2= r\hspace{2pt}dt^2 +2\dsum_{i=1}^n \big(b_i\hspace{2pt}dt\hspace{1pt} dz^i + 
        b_{\bar i}\hspace{2pt} dt\hspace{1pt}d\bar z^i\big)
     + 2\dsum_{i,j=1}^n g_{i\,\bar j}\hspace{2pt}dz^i d\bar z^j,     
\end{equation}
here $r$ stands for $g_{00}$. All coefficients $\eta_A$ vanish except $\eta_0=\eta(Z_0)$ therefore $\eta=udt$. Without loss of the generality we may assume $u=1$ for $d\eta=du\wedge dt=0$. Since $\bar\xi=\xi$ ($\xi$ is real) and $\eta(\xi)=1$ it follows that
\begin{equation}
      \xi =Z_0 + \sum_{i=1}^n \left(  a^i Z_i +a^{\bar i}Z_{\bar i} \right)\hspace{2pt},
\end{equation}
and $\bar a^i=a^{\bar i}$. In virtue of $\varphi\xi=0$ the last formula implies
\begin{equation}
         \varphi Z_0 = -\sqrt{-1}\ \sum_{i=1}^n \left( a^i Z_i
                        -a^{\bar i}Z_{\bar i} \right)\hspace{2pt}.
\end{equation}
We point out the following relations between $b_i$, $a^i$ and $r$
\begin{equation}
     r=1+2\dsum_{i,j=1}^n a^i a^{\bar j}g_{i\,\bar j}, \quad\quad 
     b_i = -\dsum_{j=1}^n a^{\bar j} g_{i\,\bar j}\hspace{2pt},                           
\end{equation}
they are consequences of $\eta(Z_A)=g(\xi,Z_A)$ and $g(\xi,\xi)=1$. 

\begin{theorem}Let $( \mathcal{M},\varphi,\xi,\eta,g)$, $\dim \mathcal{M}=2n+1\geqslant 3$  be a real-analytic almost cosymplectic manifold with K\"alerian leaves. Then
\begin{itemize}
\item[(a)] there is an open covering $(\mathcal U_\iota)_{\iota\in I}$, such that for each $\mathcal U_\iota$ there is 
a diffeomorphism   
\[
 f_\iota : U_\iota\rightarrow (-a,a)\times D',
\]
where $(-a,a)$ is an open interval, $a>0$ and $D'$ is a domain in $\mathbb C^n$; 
\item[(b)] on the set $U_\iota$, the structure $(\varphi,\xi,\eta,g)$ is described   by the following  local expressions
\begin{equation*}
\label{eksi}
\begin{array}{l}  
    \varphi\dfrac{\partial}{\partial t}=-\sqrt{-1}\ \sum_{i=1}^n \left( a^i\dfrac{\partial}{\partial z^i}
  - a^{\bar i}\dfrac{\partial}{\partial z^{\bar i}} \right), \quad  \\[+12pt]
  \varphi \dfrac{\partial}{\partial z^i}=\sqrt{-1}\ \dfrac{\partial}{\partial z^i},\quad\quad
  \varphi \dfrac{\partial}{\partial z^{\bar i}}= - \sqrt{-1}\ \dfrac{\partial}{\partial z^{\bar i }}, \quad i=1,\ldots,n,\\[+12pt]  
  \eta= dt, \quad \\[+8pt] 
    \xi = \dfrac{\partial}{\partial t} + \sum_{i=1}^n \left(a^i\dfrac{\partial}{\partial z^i}
  + a^{\bar i}\dfrac{\partial}{\partial z^{\bar i}} \right), \\[+12pt]  
    g= r\hspace{2pt}dt^2 +2\dsum_{i=1}^n \left( b_i\hspace{2pt}dt\hspace{1pt} dz^i + 
        b_{\bar i}\hspace{2pt} dt\hspace{1pt}dz^{\bar i}\right)
     + 2\dsum_{i,j=1}^n g_{i\,\bar j}\hspace{2pt}dz^i\, dz^{\bar j},     
\end{array}
\end{equation*}
\item[(c)] the coefficients $a^i$, $a^{\bar i}$, $b_i$, $b_{\bar i}$, $r$, $g_{i\,\bar j}$ are related by  
\begin{equation*}
     r=1+2\dsum_{i,j=1}^n a^i a^{\bar j}g_{i\,\bar j}, \quad\
     b_i= -\dsum_{j=1}^n a^{\bar j} g_{i\,\bar j}, \quad a^{\bar i} =\bar a^i, \quad b_{\bar i}= \bar b _i,
\end{equation*}
and $(g_{i\,\bar j})$ is $n\times n$ Hermitian matrix.
\end{itemize}
\end{theorem}

\subsection{Hermitian structure of an almost cosymplectic CR manifold}
 We start from an extension of the Levi-Civita connection to the complex connection in the complexified bundle $T_\mathbb{C}\mathcal{M}$. For a local  section $Z$ of $T_\mathbb{C}\mathcal{M}$ 
\[
 Z = \sum_{i=1}^{2n+1}f^iX_i,
\]
where $(X_1,\ldots,X_{2n+1})$ is a local frame of (real) vector fields on $\mathcal{M}$ and $f^i$ are complex valued functions we define 
\[
\begin{array}{l}
Z\mapsto \nabla Z = \nabla (\sum_{i=1}^{2n+1}f^iX_i) = \\[+12pt]
\hspace{1.5cm}=\left(\sum_{i=1}^{2n+1} df^i\otimes X_i\right) + 
   \left(\sum_{i=1}^{2n+1} f^i\nabla X_i\right),
 ×
\end{array}
\]
where $df = d(\Re f+\sqrt{-1}\Im f) =d\,(\Re f)+\sqrt{-1}d\,(\Im f)$. We have to of course 
verify that $\nabla Z$ is independent of the choice of a local real frame. This definition  follows that 
\[
\begin{array}{l}
\nabla g Z = dg\otimes Z+g\nabla Z, \quad \;\textnormal{for a complex-valued function}\; g\\[+6pt]
\nabla Z = \nabla (X+\sqrt{-1}Y) = \nabla X+\sqrt{-1}\nabla Y,
 ×
\end{array}
\]
$X = \Re Z$, $Y = \Im Z$ are real and imaginary parts of $Z$. Now if $Z$ is a section of the 
CR structure $\mathcal{D'}$ then there is a real vector field $Y$ such that 
\[
 Z=Y-\sqrt{-1}\varphi Y,\quad \eta(Y)=0,
\]
therefore
\[
\begin{array}{rcl}
 \nabla_X Z &=& \nabla_X Y-\sqrt{-1}\nabla_X\varphi Y = \\[+4pt] 
         &=&        \left(\left(\nabla_XY-\eta(\nabla_XY)\xi\right)\right) -\sqrt{-1}\varphi\nabla_XY
       +\eta(\nabla_XY)\xi +\sqrt{-1}g(A\varphi X,Y)\xi.
 ×
\end{array}
\]
Let denote $Y'= \nabla_XY-\eta(\nabla_XY)\xi$, $Y'$ is a tangential part of the derivative 
$\nabla_XY$, i.e. $\eta(Y')=0$, thus
\[
\left(\nabla_XY-\eta(\nabla_XY)\xi\right)-\sqrt{-1}\varphi\nabla_XY = 
Y'-\sqrt{-1}\varphi Y',
\]
is a $\mathcal{D'}$-component of $\nabla_X Z$. Let
$$
Z''= \text{a}\;\mathcal{D'}\text{-component of}\; \nabla_XZ,
$$
then
\[
 \nabla_XZ = Z''+\eta(\nabla_XY)\xi +\sqrt{-1}g(A\varphi X,Y)\xi.
\]
Note that $\eta(\nabla_XY)\xi = g(AX,Y)\xi$ 
according to the definition of $A$ and $\eta(Y)=0$:
\[
 \eta(\nabla_XY)\xi+\sqrt{-1}g(A\varphi X,Y)\xi = 
\left(g(AX,Y)+\sqrt{-1}g(A\varphi X,Y)\right)\xi = g(X, A \bar Z)\xi,
\]
$\bar Z$ is the complex conjugate of $Z$. Finally 
\[
 \nabla_XZ = Z'' + g(X,A\bar Z)\xi.
\]

\begin{proposition}
 The map
\[
 Z \mapsto \nabla'_XZ = \nabla_XZ-g(X,A\bar Z)\xi,
\]
defines a complex connection in the CR structure $\mathcal{D'}$ - as a connection in a complex vector bundle. Moreover this connection 
is Hermitian with respect to a Hermitian metric $H$ on $\mathcal{D'}$ defined by 
\[
 H(Z,W) = g(Z, \bar W), \quad Z,W \in \Gamma(\mathcal{D'}).
\]
\end{proposition}
\begin{proof}
To prove that $\nabla'$ is Hermitian with respect to $H$ we have to show that
\[
 XH(Z_1,Z_2) = H(\nabla'_XZ_1, Z_2)+H(Z_1,\nabla'_X Z_2),
\]
for arbitrary $X$ real vector field and arbitrary sections $Z_1$, $Z_2$ of $\mathcal{D'}$. From the definition:
\[
 XH(Z_1,Z_2) = Xg(Z_1,\bar Z_2) = g(\nabla_X Z_1, \bar Z_2)+g(Z_1, \nabla_X \bar Z_2),
\]
note that 
\[
 \begin{array}{l}
\nabla_X \bar Z_2 = \overline{\nabla_X Z_2} = \overline{\nabla'_X Z_2}+\overline{g(X,A \bar Z)\xi} = \overline{\nabla'_X Z_2}+ g(X,AZ)\xi, \\[+4pt]
g(\xi, \bar Z_2) = g(Z_1, \xi) = 0
 ×
\end{array}
\]
hence
\[
 g(\nabla_X Z_1, \bar Z_2)+g(Z_1, \nabla_X\bar Z_2)= 
g(\nabla'_XZ_1, \bar Z_2) + g(Z_1, \overline{\nabla'_X Z_2})= H(\nabla'_X Z_1, Z_2)+
H(Z_1, \nabla'_X Z_2).
\]
\end{proof}
It is very interesting topic to study the geometry of the manifold from the  point of view 
of that canonical Hermitian structure. However these problems deserve its own attention so 
we stop here - as a starting point for the further investigations. 
\section{Almost cosymplectic $(\kappa,\mu,\nu)$- spaces} 

An almost  contact metric structure of a model space $(-1,\mu,0)$, $\mu=const$ can be realized as a left-invariant structure on a  Lie group $\mathcal{G}$. If $\mathcal{G}$ is simply connected (and connected as all manifolds considered here are assumed to be connected) then $\mathcal{G}$ 
is diffeomorphic to $\mathbb{R}^{2n+1}$ and there is a frame of 
 left-invariant vector fields $(Z,X_1,\ldots,X_n,Y_1,\ldots,Y_n)$, the basis of the Lie algebra $\mathfrak{g}=\mathfrak{g}(\mu)$ of $\mathcal{G}$,  such that
\begin{equation}
 \begin{array}{l}
\xi = Z, \quad \eta(Z) =1, \quad \eta(X_i)=\eta(Y_j)=0, \; i,j=1,\ldots,n ,\\[+4pt]
\varphi X_i = Y_i,\quad \varphi Y_i = - X_i, \; i=1,\ldots,n ,
 ×
\end{array}
\end{equation}
and the frame is orthonormal.  Let $(t,x^1,\ldots,x^n,y^1,\ldots,y^n)$ be a global chart on $\mathbb{R}^{2n+1}$ then  the vector fields $X_i$'s, $Y_j$'s are described explicitly as follows (in all cases $\xi=\partial/\partial t$ and $i=1,\ldots,n$) \cite{DO3}:
\begin{enumerate}
 \item $|\mu| < 2$, $\omega=\sqrt{1-\mu^2/4}$:
\begin{equation*}
 \begin{array}{rcl}
X_i &=& (\cosh(\omega t)+\dfrac{\sinh(\omega t)}{\omega})\dfrac{\partial}{\partial x^i} 
      - \dfrac{\mu \sinh(\omega t)}{2\omega}\dfrac{\partial}{\partial y^i}, \\[+12pt]
Y_i &=& \dfrac{\mu \sinh(\omega t)}{2\omega}\dfrac{\partial}{\partial x^i} 
      + (\cosh(\omega t)-\dfrac{\sinh(\omega t)}{\omega})\dfrac{\partial}{\partial y^i},   
×
\end{array}
\end{equation*}
\item $|\mu| =2$ ($\omega=0$):
\begin{equation*}
 X_i = (1+t)\dfrac{\partial}{\partial x^i} -\varepsilon\,t\,\dfrac{\partial}{\partial y^i}, \quad
Y_i = \varepsilon\,t\, \dfrac{\partial}{\partial x^i} + (1-t)\dfrac{\partial}{\partial y^i},\quad 
 \varepsilon = \mu/2 = \pm 1,
\end{equation*}
\item $|\mu|>2$, $\omega=\sqrt{-1+\mu^2/4}$:
\begin{equation*}
 \begin{array}{rcl}
X_i &=& (\cos(\omega t)+ \dfrac{\sin(\omega t)}{\omega})\dfrac{\partial }{\partial x^i} 
      - \dfrac{\mu\sin(\omega t)}{2\omega}\dfrac{\partial}{\partial y^i},\\[+12pt]
Y_i &=& \dfrac{\mu\sin(\omega t)}{2\omega}\dfrac{\partial}{\partial x^i} + 
      (\cos(\omega t)-\dfrac{\sin(\omega t)}{\omega})\dfrac{\partial}{\partial y^i},
 ×
\end{array}
\end{equation*}
\end{enumerate}

As we see the local descriptions are quite different depending on the value of $\mu$. 
Note that algebras $\mathfrak{g}(\pm 2)$ are 'limits' (two-sided) 
\[
 \mathfrak{g}(\pm 2) = \lim\limits_{\mu \rightarrow \pm 2}\mathfrak{g}(\mu), \quad \mu\neq \pm 2.
\]
 Nevertheless the 
commutators of the vector fields can be described in a unique manner in all cases:
\begin{equation}
\label{munu-comm}
 \mathfrak{g}(\mu):\quad [\xi,X_i] = X_i-\dfrac{\mu}{2}Y_i, \quad [\xi,Y_i] = \dfrac{\mu}{2}X_i-Y_i,
\end{equation}
the other commutators vanish identically. 

On the base of the relations (\ref{munu-comm})  here we will provide a different local representation. Let again $\mathcal{M} = \mathbb{R}^{2n+1}$, $\mathcal{M}\ni p = (t,x^1, \ldots,x_n,y^1,\ldots,y_n)$
and we set 
\begin{equation}
\label{munu-xiCR}
\begin{array}{l}
\xi = \dfrac{\partial}{\partial t} + \sum\limits_{i=1}^n(-x^i-\dfrac{\mu}{2}y^i)\dfrac{\partial}{\partial x^i}
 + \sum\limits_{i=1}^n(\dfrac{\mu}{2}x^i+y^i)\dfrac{\partial}{\partial y^i},\\
X_i =\dfrac{\partial}{\partial x^i},\qquad Y_i =\dfrac{\partial}{\partial y^i}, \quad i=1,\ldots,n   
×
\end{array}
\end{equation}
Direct computations show that such defined vector fields  satisfy (\ref{munu-comm}). Now, let 
identify $\mathbb{R}^{2n+1} \cong \mathbb R \times \mathbb{C}^n$ in the way that 
\begin{equation}
\label{munu-CRcoord}
\begin{array}{l}
 p = (t,x^1,\ldots,x^n,y^1,\ldots,y^n) = (t, z^1,\ldots, z^n) \in \mathbb{R}\times \mathbb{C}^n, \\[+4pt]
  z^1 = x^1+\sqrt{-1}\,y^1, \ldots,\; z^n = x^n+\sqrt{-1}\,y^n,
\end{array}
\end{equation}
the vector fields 
\begin{equation}
\label{munu-Zframe}
\begin{array}{rcl}
Z_0 &=& \dfrac{\partial}{\partial t},  \\[+6pt]
Z_i &=& \dfrac{1}{2}(\dfrac{\partial}{\partial x^i}-\sqrt{-1}\;\dfrac{\partial}{\partial y^i})=
\dfrac{\partial}{\partial  z^i},
\quad i=1,\ldots,n \\[+12pt]
Z_{\bar i} &=& \dfrac{1}{2}(\dfrac{\partial}{\partial x^i}+\sqrt{-1}\dfrac{\partial}{\partial y^i})=
\dfrac{\partial}{\partial \bar z^i}, 
\quad i =1,\ldots,n 
×
\end{array}
\end{equation}
form a frame (in our case global) of vector fields of the complexified tangent bundle $T_{\mathbb{C}}\mathcal{M}$. According to (\ref{munu-xiCR}), (\ref{munu-CRcoord}), (\ref{munu-Zframe}) we have
\[ 
\begin{array}{l}
\xi =\dfrac{\partial}{\partial t}+ \sum\limits_{i=1}^n(-x^i-\dfrac{\mu}{2}y^i)\dfrac{\partial}{\partial x^i} 
+ \sum\limits_{i=1}^n(\dfrac{\mu}{2}x^i+y^i)\dfrac{\partial}{\partial y^i} = \\
= \dfrac{\partial}{\partial t} +\dfrac{1}{2}\sum\limits_{i=1}^n(-(z^i+\bar z^i)-\dfrac{\mu}{2}\dfrac{z^i -\bar z^i}{\sqrt{-1}})(\dfrac{\partial}{\partial z^i}+\dfrac{\partial}{\partial \bar z^i}) + \\
+ \dfrac{1}{2}\sum\limits_{i=1}^n(\dfrac{\sqrt{-1}\mu}{2}(z^i+\bar z^i) + (z^i-\bar z^i))(\dfrac{\partial}{\partial z^i}-\dfrac{\partial}{\partial \bar z^i}).
\end{array}
\]
Providing similar computations for the (complexified) tensor fields $\varphi$, $\eta$, and $g$ we obtain such expressions:  
\begin{equation*}
 \begin{array}{l}
 \eta = dt, \\
 \xi = \dfrac{\partial}{\partial t} + 
  \sum\limits_{i=1}^n(-\bar z^i +\dfrac{\sqrt{-1}\mu}{2}z^i)\dfrac{\partial }{\partial z^i} 
+\sum\limits_{i=1}^n (-z^i-\dfrac{\sqrt{-1}\mu}{2}\bar z^i)\dfrac{\partial}{\partial \bar z^i}, \\[+12pt] 
\sqrt{-1}\;\varphi\;\dfrac{\partial}{\partial t} =\sum\limits_{i=1}^n (-\bar z^i+\dfrac{\sqrt{-1}\mu}{2}z^i)\dfrac{\partial}{\partial z^i} +\sum\limits_{i=1}^n(z^i+\dfrac{\sqrt{-1}\mu}{2}\bar z^i)\dfrac{\partial}{\partial \bar z^i},\\[+12pt]
\varphi \dfrac{\partial}{\partial z^i} = \sqrt{-1}\dfrac{\partial}{\partial z^i},\qquad
\varphi \dfrac{\partial}{\partial \bar z^i} = -\sqrt{-1}\dfrac{\partial}{\partial \bar z^i},\qquad i=1,\ldots,n ,\\[+12pt]
ds^2 = r\,dt^2 + 2\sum\limits_{i=1}^n (z^i+\dfrac{\sqrt{-1}\mu}{2}\bar z^i) dt dz^i 
+ 2\sum\limits_{i=1}^n(\bar z^i-\dfrac{\sqrt{-1}\mu}{2}z^i) dt d\bar z^i + 2\sum\limits_{i=1}^n dz^i d\bar z^i,\\
r= 1+2\sum\limits_{i=1}^n|z^i+\dfrac{\sqrt{-1}\mu}{2}\bar z^i|^2.
\end{array}
\end{equation*}
Recently D. Perrone has classified  Riemannian homogeneous simply connected almost cosymplectic three-folds under an assumption that there is a group of isometries acting transitively and leaving the form $\eta$ invariant \cite{Perrone}. 
\begin{theorem}\cite{Perrone} Let $(\mathcal{M},\varphi,\xi,\eta,g)$ be a be a simply connected homogeneous
\footnote{It is assumed that  $\eta$ is invariant} almost cosymplectic three-manifold. Then either $\mathcal{M}$ is a Lie group $G$ equipped with a left invariant almost cosymplectic structure, or a Riemannian product of type $\mathbb{R}\times \mathcal{N}$, where $\mathcal{N}$ is a simply connected K\"ahler surface of constant curvature.
\end{theorem}
From the list provided in \cite{Perrone} we are interested only in the case of non-cosymplectic three-manifolds and unimodular Lie groups as all Lie groups corresponding to the algebras $\mathfrak{g}(\mu)$ are unimodular:
\begin{itemize}
\item the universal covering $\tilde E(2)$ of the group of rigid motions of Euclidean 2-space, when $p > 0$,
\item the group $E(1,1)$ of rigid motions of Minkowski 2-space when $p < 0$,
\item the Heisenberg group $H^3$ when $p = 0$,
\end{itemize}
here $p$ is a metric invariant of the classification defined by
\[
 p = \|\mathcal{L}_\xi h\|-2\|h\|^2.
\]
\begin{corollary}
 Let $(\mathcal{M}^3,\varphi,\xi,\eta,g)$ be a simply connected 3-dimensional almost cosymplectic \\ $(-1,\mu,0)$-space. Then $\mathcal{M}^3$ is a Lie group equipped with appropriate left invariant almost cosymplectic structure and 
\begin{itemize}
 \item $\mathcal{M}^3$ is a universal covering of the group of rigid motions of  Euclidean 2-plane for  $|\mu| > 2$,
\item  $\mathcal{M}^3$ is a group of rigid motions of Minkowski 2-plane for $|\mu| <2 $,
\item $\mathcal{M}^3$ is a Heisenberg group for $|\mu| =2 $
\end{itemize}

\end{corollary}

\end{document}